\documentclass[12pt,reqno]{amsart}
\usepackage{amsmath,amsthm,amssymb,amsfonts,amscd}
\usepackage{mathrsfs}
\usepackage{bbm}
\usepackage{bbding}
\usepackage[backref=page]{hyperref}
\usepackage{geometry}
\usepackage{color}
\usepackage{xcolor}

\usepackage{picture,epic}
\usepackage{tikz}

\numberwithin{equation}{section}

\theoremstyle{plain}
\newtheorem{theorem}{Theorem}[section]
\newtheorem{lemma}[theorem]{Lemma}

\newtheorem{proposition}[theorem]{Proposition}

\theoremstyle{definition}
\newtheorem{conjecture}[theorem]{Conjecture}

\theoremstyle{remark}
\newtheorem{remark}[theorem]{Remark}

\renewcommand{\Re}{\operatorname{Re}}
\renewcommand{\Im}{\operatorname{Im}}
\newcommand{\vol}{\operatorname{vol}}

\newcommand{\sym}{\operatorname{sym}}

\newcommand{\GL}{\operatorname{GL}}
\newcommand{\SL}{\operatorname{SL}}

\newcommand{\dd}{\mathrm{d}}

\makeatletter
\def\@tocline#1#2#3#4#5#6#7{\relax
  \ifnum #1>\c@tocdepth 
  \else
    \par \addpenalty\@secpenalty\addvspace{#2}%
    \begingroup \hyphenpenalty\@M
    \@ifempty{#4}{%
      \@tempdima\csname r@tocindent\number#1\endcsname\relax
    }{%
      \@tempdima#4\relax
    }%
    \parindent\z@ \leftskip#3\relax \advance\leftskip\@tempdima\relax
    \rightskip\@pnumwidth plus4em \parfillskip-\@pnumwidth
    #5\leavevmode\hskip-\@tempdima
      \ifcase #1
       \or\or \hskip 1em \or \hskip 2em \else \hskip 3em \fi%
      #6\nobreak\relax
    \hfill\hbox to\@pnumwidth{\@tocpagenum{#7}}\par
    \nobreak
    \endgroup
  \fi}
\makeatother

\begin{document}

\title
{Joint value distribution of Hecke--Maass forms}
\author{Shenghao Hua}
\address{Data Science Institute and School of Mathematics \\ Shandong University \\ Jinan \\ Shandong 250100 \\China}
\email{shenghao.hua@epfl.ch}
\curraddr
{\itshape EPFL-SB-MATH-TAN
Station 8,
1015 Lausanne, Switzerland}

\author{Bingrong Huang}
\address{Data Science Institute and School of Mathematics \\ Shandong University \\ Jinan \\ Shandong 250100 \\China}
\email{brhuang@sdu.edu.cn}

\author{Liangxun Li}
\address{Data Science Institute and School of Mathematics \\ Shandong University \\ Jinan \\ Shandong 250100 \\China}
\email{lxli@mail.sdu.edu.cn}

\date{\today}

\begin{abstract}
  In this paper, we formulate a conjecture on joint distribution of Hecke--Maass cusp forms. To support our conjecture, we prove two conditional results on joint moments of two Hecke--Maass cusp forms, which confirms statistical independence of orthogonal cusp forms.
\end{abstract}



\thanks{This work was supported by  the National Key R\&D Program of China (No. 2021YFA1000700) and NSFC (No. 12031008).}

\maketitle

\section{Introduction} \label{sec:Intr}

The value distribution of eigenfunctions in the semiclassical limit is one of the main problems in analytic number theory and quantum chaos.
Let $\mathbb{H}=\{z=x+iy:x\in\mathbb{R}, y>0\}$ be the upper half plane with the hyperbolic measure $\dd\mu z = \dd x\dd y/y^2$ and $\Gamma=\SL_2(\mathbb{Z})$ the modular group.
The modular surface $\Gamma\backslash\mathbb{H}$ is a noncompact, but finite area hyperbolic surface.
For automorphic functions on $\Gamma\backslash\mathbb{H}$ we have the inner product which is defined by $\langle f,g\rangle := \int_{\Gamma\backslash\mathbb{H}} f(z)\overline{g(z)} \dd\mu z$.
The Laplacian is given by $\Delta=-y^2 (\partial^2/\partial x^2+\partial^2/\partial y^2)$, which has both discrete and continuous spectra.
The discrete spectrum consists of the constants and the space of cusp forms, for which we can take an orthonormal basis $\{u_j\}$  of Hecke--Maass forms, which are real valued joint eigenfunctions of both the Laplacian and all the Hecke operators, and decay exponentially at the cusp of the modular domain.
The modular surface  carries a further symmetry induced by the
reflection $z\mapsto -\bar{z}$ of $\mathbb H$ and our $u_j$'s are either even or odd with respect to this symmetry.
Berry \cite{Berry}  suggested  that eigenfunctions for chaotic systems are modeled by random waves, it is believed that eigenfunctions on a compact hyperbolic surface have a Gaussian value distribution as the eigenvalue tends to infinity, and the moments of an $L^2$-normalized eigenfunction should be given by the Gaussian moments.
More precisely, we have the following Gaussian moments conjecture. (See e.g. Humphries \cite[Conjecture 1.1]{Humphries}.)

\begin{conjecture}\label{conj:1}
  Fix any $\psi\in \mathcal{C}_c^\infty(\Gamma\backslash \mathbb{H}) $.
  Let $f$ be a Hecke--Maass cusp form with the spectral parameter $t_f$. Assume
  \[ \frac{1}{\vol(\Gamma\backslash \mathbb{H})}   \int_{\Gamma\backslash \mathbb{H}} |f(z)|^2\dd\mu z = 1. \]
  Then for any $n\in\mathbb{Z}_{\geq1}$, we have
  \begin{equation}\label{eqn:nthmoment}
     \int_{\Gamma\backslash \mathbb{H}} \psi(z) f(z)^n \dd\mu z = c_n  \int_{\Gamma\backslash \mathbb{H}}\psi(z)\dd\mu z + o(1),
  \end{equation}
  as $t_f\rightarrow\infty$. Here
  \begin{equation}\label{eqn:cn}
    c_n = \left\{ \begin{array}{ll}
                    (2n-1)!!,  & \textrm{if $n$ is even,} \\
                    0,  & \textrm{if $n$ is odd.}
                  \end{array} \right.
  \end{equation}
\end{conjecture}

\begin{remark}
  Let $X\sim \mathcal{N}(0,1)$ be a standard Gaussian normal distribution. Then for any $n\in\mathbb{N}$, we have $\operatorname{E}(X^n)=c_n$.
\end{remark}

For $n=2$, this is the well-known quantum unique ergodicity (QUE) conjecture of Rudnick and Sarnak \cite{RudnickSarnak1994behaviour}. QUE was solved by breakthrough papers of Lindenstrauss \cite{lindenstrauss2006invariant} and Soundararajan \cite{soundararajan2010quantum}.

For  $n=3$, Huang \cite{Huang} recently solved  this cubic moment problem, extending the result of Watson \cite{watson2008rankin} for the case
$\psi(z)=1$ for all $z\in \mathbb{H}$. 

For $n=4$, this is the $L^4$-norm problem. Unconditionally, only nontrivial upper bounds are known for the modular surface. See e.g. \cite{HumphriesKhan, Ki}. Buttcane and Khan \cite{ButtcaneKhan2017} proved the asymptotic formula when $\psi(z)=1$ for all $z\in \mathbb{H}$,
by assuming the generalized Lindel\"of Hypothesis (GLH).

\medskip

In this paper, our main interest is the joint value distribution of  orthogonal automorphic forms in the semiclassical limit.
Let $J\geq2$ be a fixed integer.
Let $f_j$, $1\leq j\leq J$, be pairwisely orthogonal real-valued Hecke--Maass cusp forms for $\SL_2(\mathbb{Z})$, i.e., $\langle f_i,f_j\rangle=0$ for all $1\leq i\neq j\leq J$.
Assuming
\begin{equation}\label{eqn:L2normalization}
  \frac{1}{\vol(\Gamma\backslash \mathbb{H})}   \int_{\Gamma\backslash \mathbb{H}} |f_j(z)|^2\dd\mu z
  = 1, \quad 1\leq j\leq J.
\end{equation}
For positive integers $a_j$, $1\leq j\leq J$, it is interesting to consider values of  $ f_j^{a_j}$ against a fixed observable $\psi$.  We formulate the following conjecture, which predicts that the values of distinct Hecke--Maass cusp forms should behave like  independent random waves.

\begin{conjecture}\label{conj:2}
  With $f_j$ and $a_j$ as above. Then $\{f_j^{a_j}\}_{j=1}^{J}$ are statistically independent; that is,
  for any $\psi\in \mathcal{C}_c^\infty(\Gamma\backslash \mathbb{H})$, we have
  \[
     \int_{\Gamma\backslash \mathbb{H}} \psi(z) \prod_{j=1}^{J} f_j(z)^{a_j}  \dd\mu z
    = \prod_{j=1}^{J} c_{a_j} \int_{\Gamma\backslash \mathbb{H}}\psi(z)\dd\mu z + o(1)
  \]
  as $\min(t_{f_1},\ldots,t_{f_J})$ goes to infinity.
  Here $c_a$ is defined as in \eqref{eqn:cn}.
\end{conjecture}


The first nontrivial case is $(a_1,a_2)=(1,1)$, which is the off-diagonal matrix element of $\textrm{Op}(\psi)$. Assuming the generalized Lindel\"of Hypothesis (GLH), we can prove our conjecture in this case.

For the case $(a_1,a_2)=(2,1)$, this is closely related to the cubic moment problem of Maass forms, i.e. the case $n=3$ in Conjecture \ref{conj:1}.
We prove our conjecture in this case under GLH.

\begin{theorem}\label{thm:psi f^2g}
    Let $f,g$ be two Hecke--Maass cusp forms with the spectral parameters $t_f,t_g$, normalized as in \eqref{eqn:L2normalization}.
    Fix $\psi\in \mathcal{C}_c^\infty(\Gamma\backslash \mathbb{H})$.
    Let $\varepsilon$ be a small positive number.
    When $2t_f\leq t_g-t_g^\varepsilon$, we have
    \begin{equation}
       \int_{\Gamma\backslash\mathbb{H}}
       \psi(z)f^2(z)g(z)\dd \mu z
         \ll_{\psi, A} t_g^{-A}, \quad  \textrm{for any $A>0$}.
    \end{equation}
   When $t_g-t_g^\varepsilon<2t_f$,
   assuming GLH, we have
    \begin{equation}
    \begin{split}
        \int_{\Gamma\backslash\mathbb{H}}
        \psi(z)f^2(z)g(z)\dd \mu z
        &=
        \int_{\Gamma\backslash\mathbb{H}}
        \psi(z)g(z)\dd \mu z +O_{\psi, \varepsilon}\Big(t_f^{-\frac{1}{4}+\varepsilon} (1+|2t_f-t_g|)^{-\frac{1}{4}} \Big).
        \end{split}
    \end{equation}
   Moreover, if $t_g$ goes to infinity, then we have $\langle \psi g, f^2 \rangle \sim 0$ under GLH.
\end{theorem}


\begin{remark}
  By the method in Huang \cite{Huang}, one can at least deal with the case when $t_{g} - t_{f} = O(1)$ unconditionally.
  It will be nice to work out what exactly one can prove unconditionally in this problem.
  Note that in Theorem \ref{thm:psi f^2g}, when $t_g-2t_f =O(1)$, the upper bounds are weaker than other cases, due to the conductor drop phenomenon.
\end{remark}

It is interesting to see whether one can obtain the conjectured main term in the asymptotic formula in Conjecture \ref{conj:2} when it does not vanish.
The simplest case is when $J=2$ and $(a_1,a_2)=(2,2)$.
But it seems too hard to deal with an observable $\psi\in \mathcal{C}_c^\infty(\Gamma\backslash\mathbb{H})$.
For the case $(a_1,a_2)=(2,2)$ and $\psi(z)=1$ for all $z\in \mathbb{H}$, this is closely related to the $L^4$-norm problem of Maass forms. We are able to prove the following theorem, which only needs the larger spectral parameter to tend to infinity.

\begin{theorem}\label{thm:22}
   Assume the generalized Riemann Hypothesis (GRH) and the generalized Ramanujan Conjecture (GRC).
 Let $f,g$ be two orthogonal Hecke--Maass cusp forms with the spectral parameters $t_f,t_g$, normalized as in \eqref{eqn:L2normalization}.
  Then we have
  \[
    \frac{1}{\vol(\Gamma\backslash \mathbb{H})} \int_{\Gamma\backslash \mathbb{H}} f(z)^2 g(z)^2 \dd\mu z  = 1 + O_{\varepsilon}\big( (\log (t_f + t_g))^{-1/4+\varepsilon} \big)
  \]
  as $\max(t_f,t_g)$ goes to infinity. 
\end{theorem}


\begin{remark}
  Our method can deal with $\langle f_1f_2, f_3f_4 \rangle$ whenever not all of $f_i$'s are the same.
  As in Huang \cite{Huang2024}, Theorem \ref{thm:22} has an application to the non-vanishing problem of the triple product $L$-functions $L(1/2,f\times g\times u_j)$.
\end{remark}

To estimate $\int_{\Gamma\backslash \mathbb{H}} f(z)^2 g(z)^2 \dd\mu z$, it is natural to view it as $\langle fg,fg\rangle$. By Parseval's identity, we have
\[
  \langle fg,fg\rangle = \sum_{j\geq1} |\langle fg,u_j\rangle|^2 + \frac{1}{4\pi} \int_{\mathbb{R}} |\langle fg,E(\cdot,1/2+it)\rangle|^2 \dd t  ,
\]
where $E(z,1/2+it)$ is the Eisenstein series. Here the contribution from the constant function vanishes since $\langle f,g\rangle=0$.
It is not hard to see the contribution from the Eisenstein series will be small if we assume GRH. Then we want to show that the contribution form the cusp forms will be $1+o(1)$. To do this, one may use the Watson's formula to reduce this to the first moment of the triple product $L$-functions $L(1/2,f\times g\times u_j)$.
Bounding this effectively is still hard in our problem even if we assume GRH and GRC.
This may be more complicated than Buttcane and Khan's work \cite{ButtcaneKhan2017}, where they  deal with the first moment of $L(1/2, u_j) L(1/2,\sym^2 f \times u_j)$ under GLH. However, we can get $1+o(1)$ heuristically. One good thing in this approach is that we obtain an identity between $\langle fg,fg\rangle$  and the first moment of $L$-functions.
But we do not know how to make use of GRH and GRC.

To prove Theorem \ref{thm:22},
our novelty is to view $\int_{\Gamma\backslash \mathbb{H}} f(z)^2 g(z)^2 \dd\mu z$ as $\langle f^2, g^2 \rangle$.
By Parseval's identity, we have
\[
  \frac{\langle f^2,g^2\rangle}{\vol(\Gamma\backslash \mathbb{H})} = 1 + O\Big(\sum_{j\geq1} |\langle f^2,u_j\rangle \langle g^2,u_j\rangle|  + \int_{\mathbb{R}} |\langle f^2,E(\cdot,1/2+it)\rangle \langle g^2,E(\cdot,1/2+it)\rangle| \dd t \Big ).
\]
Since we do not know the explicit formula of $\langle f^2,u_j\rangle$ in terms of $L$-functions, we take the absolute values. We lose a lot from this step. But we already see the main term. So we only need to show that the contributions from the cusp forms and Eisenstein series should be small. For this, we will use Watson's formula to reduce the problem  to bounding fractional moments of $L$-functions, and then apply Soundararajan's method  in \cite{soundararajan2009moments} to deal with the moments under GRH and GRC.


In the end of this introduction, we mention that a statistical independence result of Hecke eigenfunctions has been proved in Kurlberg and Rudnick \cite{KR01b} for a well studied model in quantum chaos -- the quantization of a hyperbolic linear map of the torus named ``quantum cat map''.

\medskip
\textbf{Structure of this paper.}
The rest of this paper is organized as follows.
In \S \ref{sec:21}, we prove Theorem \ref{thm:psi f^2g} by using the theory of $L$-functions.
In \S \ref{sec:22}, we prove Theorem \ref{thm:22} by assuming Proposition \ref{prop:mixedmonents1} on moments of $L$-functions.
In \S \ref{sec:fractional-moment}, a variant of Soundararajan's method is applied to give an almost sharp upper bound of certain fractional moment of $L$-functions, which proves Proposition \ref{prop:mixedmonents1}.

\section{Proof of Theorem \ref{thm:psi f^2g}} \label{sec:21}

In this section, we prove Theorem \ref{thm:psi f^2g}.
By Selberg's spectral decomposition of $L^2(\Gamma\backslash\mathbb{H})$, we know that to consider a fixed $\psi\in \mathcal C_c^\infty(\Gamma\backslash\mathbb{H})$ we need to deal with the constant function, the Hecke--Maass cusp forms, and the Eisenstein series.
We need the following lemmas.
\begin{lemma}\label{lemma: <f^2, g>}
   Assume GLH.
 Let $f,g$ be the same as in Theorem \ref{thm:psi f^2g}.
Then when $t_g\geq 2t_f$, we have
   \begin{equation}\label{eqn:f^2g1a}
        \int_{\Gamma\backslash\mathbb{H}}
        f^2(z)g(z)\dd \mu z \ll {t_g}^{-\frac{3}{4}+o(1)}
        (t_g-2t_f+1)^{-\frac{1}{4}}
        e^{-\frac{\pi}{2}(t_g-2t_f)};
    \end{equation}
    and when $t_g\leq 2t_f$, we have
  \begin{equation}\label{eqn:f^2g1b}
        \int_{\Gamma\backslash\mathbb{H}}
        f^2(z)g(z)\dd \mu z \ll {t_g}^{-\frac{1}{2}}
        {t_f}^{-\frac{1}{4}+o(1)}
        (2t_f-t_g+1)^{-\frac{1}{4}}.
    \end{equation}
  Furthermore, if $2t_f\leq t_g-t_g^\varepsilon$, then we have
$\langle 1, f^2g\rangle\ll  e^{-\frac{\pi}{2}(t_g-2t_f)+t_g^{o(1)}}$ unconditionally.
\end{lemma}

\begin{lemma}\label{lemma:ef^2g}
   Let $f,g$ be the same as in Theorem \ref{thm:psi f^2g}.
   Let $u_k$ be a Hecke--Maass cusp form  with the spectral parameter $t_k>1$, and
    $E_t(z):=E(z,1/2+it)$ the Eisenstein series with $t\in \mathbb{R}$.
    Assume  $t_k\ll \max(t_f, t_g)^{o(1)}$ and $t\ll \max(t_f, t_g)^{o(1)}$.
    When $t_g-t_g^{\varepsilon}\geq 2t_f$, we have
    \begin{equation}\label{eqn:ukf2g1}
      \int_{\Gamma\backslash\mathbb{H}}
      u_k(z)f^2(z)g(z)\dd \mu z\ll
      e^{-\frac{\pi}{2}(t_g-2t_f)+t_g^{o(1)}},
    \end{equation}
    and
    \begin{equation}\label{eqn:etf2g1}
        \int_{\Gamma\backslash\mathbb{H}}
        E_t(z)f^2(z)g(z)\dd \mu z\ll
        e^{-\frac{\pi}{2}(t_g-2t_f)+t_g^{o(1)}}.
    \end{equation}
    When $t_g-t_g^{\varepsilon}\leq 2t_f$, by assuming GLH, we have
    \begin{equation}\label{eqn:ukf2g2}
      \int_{\Gamma\backslash\mathbb{H}}
      u_k(z)f^2(z)g(z)\dd \mu z
       =
       \langle u_kg, 1\rangle
       +O\Big(t_f^{-\frac{1}{4}+o(1)} (1+|2t_f-t_g|)^{-\frac{1}{4}}\Big),
    \end{equation}
    and
    \begin{equation}\label{eqn:etf2g2}
        \int_{\Gamma\backslash\mathbb{H}}
        E_t(z)f^2(z)g(z)\dd \mu z
        \ll
        t_f^{-\frac{1}{4}+o(1)} (1+|2t_f-t_g|)^{-\frac{1}{4}}.
    \end{equation}
\end{lemma}

Now we are ready to prove Theorem \ref{thm:psi f^2g} by using Lemmas \ref{lemma: <f^2, g>} and \ref{lemma:ef^2g}.

\begin{proof}[Proof of Theorem \ref{thm:psi f^2g}]
By using Parseval's formula, we have
\begin{multline}\label{eqn:Parseval_psi f^2g}
    \int_{\Gamma\backslash\mathbb{H}}
    \psi(z)f^2(z)g(z)\dd \mu z=
    \langle \psi , \frac{3}{\pi}\rangle\langle 1, f^2g\rangle\\
    +
    \sum_{k\geq 1}\langle \psi , u_k\rangle \langle u_k, f^2g \rangle
    +
    \frac{1}{4\pi}\int_{\mathbb{R}}\langle \phi , E_{t}\rangle \langle E_{t}, f^2g\rangle \dd t.
\end{multline}
Noticing the fact
\begin{multline}
\langle \psi , u_k\rangle=\frac{1}{(1/4+t_k^2)^l}\langle \psi , \Delta ^l u_k\rangle
=\frac{1}{(1/4+t_k^2)^l}\langle \Delta ^l \psi , u_k\rangle
\\
\ll\frac{1}{(1/4+t_k^2)^l}
\int_{\Gamma\backslash\mathbb{H}}|u_k(z)||\Delta^l \psi(z)|\dd \mu z,
\end{multline}
\begin{multline}
\langle \psi , E_t\rangle=\frac{1}{(1/4+t^2)^l}\langle \psi , \Delta ^l E_t\rangle
=\frac{1}{(1/4+t^2)^l}\langle \Delta ^l \psi ,  E_t\rangle
\\
\ll\frac{1}{(1/4+t^2)^l}
\int_{\Gamma\backslash\mathbb{H}}
|E_t(z)||\Delta^l \psi(z)|\dd \mu z,
\end{multline}
and the sup norm bound for Eisenstein series
(see \cite[Theorem 1.1]{Young2018supnorm})
\begin{equation}
\max_{z\in \Omega}|E_t(z)|\ll_{\Omega, \varepsilon} (1+|t|)^{\frac{3}{8}+\varepsilon},
\end{equation}
where $\Omega$ is a fixed compact subset in $\mathbb{H}$.
Using the Cauchy--Schwarz inequality, we have
\begin{equation}
\langle \psi , u_k\rangle\ll_{\psi, A} t_k^{-A}, \quad \langle \psi , E_t\rangle\ll_{\psi, A}(1+|t|)^{-A} \, \text{for any} \,A>0.
\end{equation}
By using Parseval's formula and Watson's formula, together with the convexity bounds of central $L$-values,
we have
\begin{equation}\label{eqn:<u_k, f^2g>_trivialbound}
 \langle u_k, f^2g \rangle\ll (t_ft_gt_k)^B \quad \text{and} \quad \langle E_{t}, f^2g\rangle\ll ((1+|t|)t_ft_g)^B,
\end{equation}
for some absolute constant $B>0$. Then using \eqref{eqn:Parseval_psi f^2g} and \eqref{eqn:<u_k, f^2g>_trivialbound}, we know the contributions from $u_k$ and $E_t$ with $t_k\geq \max(t_f, t_g)^{o(1)}$ and $|t|\geq \max(t_f, t_g)^{o(1)}$ are both negligibly small.
Hence we have
\begin{multline*}
    \int_{\Gamma\backslash\mathbb{H}}
    \psi(z)f^2(z)g(z)\dd \mu z=
    \langle \psi , \frac{3}{\pi}\rangle\langle 1, f^2g\rangle
    + \delta_{t_g\ll t_f^{o(1)}}\frac{3}{\pi}\langle \psi , g\rangle  \langle 1, f^2 g^2\rangle\\
    +
    \sum_{\substack{t_k\ll \max(t_f, t_g)^{o(1)}\\ u_k\neq \sqrt{\frac{3}{\pi}}g}}\langle \psi , u_k\rangle \langle u_k, f^2g \rangle
    +
    \frac{1}{4\pi}\int_{|t|\ll \max(t_f, t_g)^{o(1)}}\langle \psi , E_{t}\rangle \langle E_{t}, f^2g\rangle \dd t\\
    +O(\max(t_f, t_g)^{-A}).
\end{multline*}
Then by using Lemma \ref{lemma: <f^2, g>}  and Lemma \ref{lemma:ef^2g} in different cases, we complete the proof of Theorem \ref{thm:psi f^2g}.
\end{proof}

\subsection{Proof of Lemma \ref{lemma: <f^2, g>}}
By using Watson's formula and the factorization
\begin{equation*}
    \Lambda(1/2, f\times f\times g)=\Lambda(1/2, \sym^2f \times g)\Lambda(1/2, g),
\end{equation*}
we have
\begin{equation}
    |\langle f^2 , g\rangle|^2 \ll \frac{\Lambda(1/2, \sym^2f\times g)\Lambda(1/2, g)}{\Lambda(1, \sym^2f)^2\Lambda(1, \sym^2g)}.
\end{equation}
From Lapid's theorem \cite{Lapid} we know the nonnegativity of central $L$-values, then using Stirling's formula, we have
\begin{equation}
    \langle f^2 , g\rangle\ll \frac{L(1/2, \sym^2f\times g)^{1/2}L(1/2, g)^{1/2}\exp\Big(-\frac{\pi}{2}(|t_f+\frac{t_g}{2}|+|t_f-\frac{t_g}{2}|-2t_f)\Big)}
    {L(1,\sym^2f)L(1, \sym^2g)^{1/2}{t_g}^{1/2}\prod_{\pm}(1+|t_g\pm2t_f|)^{1/4}}.
\end{equation}
We have \eqref{eqn:f^2g1a} and \eqref{eqn:f^2g1b} in different cases, under GLH.

\subsection{The Rankin--Selberg theory and Watson's formula}
Before proving Lemma \ref{lemma:ef^2g}, we introduce some useful formulas as follow.
According to the Rankin--Selberg theory, we have
\begin{equation}
    \langle u_kg, E_{\tau}\rangle
    =\frac{\rho_k(1)\rho_g(1)\Lambda(1/2+i\tau, g\times u_j)}{2\Lambda(1+2i\tau)},
\end{equation}
\begin{equation}
    \langle E_tg , u_j\rangle
    =\frac{\rho_j(1)\rho_g(1)\Lambda(1/2+it, g\times u_j)}{2\Lambda(1+2it)},
\end{equation}
\begin{equation}
    \langle E_tg , E_{\tau}\rangle
    =\frac{\rho_t(1)\rho_g(1)\Lambda(1/2+i\tau+it, g)\Lambda(1/2+i\tau-it, g)}{2\Lambda(1+2i\tau)},
\end{equation}
where $\rho_j(1)$, $\rho_k(1)$, $\rho_t(1)$ and $\rho_g(1)$ are normalized coefficients which satisfy
$|\rho_j(1)|^2=\frac{\cosh(\pi t_j)}{L(1, \sym^2u_j)}$,
$|\rho_k(1)|^2=\frac{\cosh(\pi t_k)}{L(1, \sym^2u_k)}$,
$|\rho_t(1)|^2=\frac{\cosh(\pi t)}{|\zeta(1+2it)|^2}$,
$|\rho_g(1)|^2=\frac{\pi}{3}\frac{\cosh(\pi t_g)}{L(1, \sym^2g)}$.
Thus
\begin{equation}\label{eqn:|<u_kg, E>|^2}
   |\langle u_kg , E_\tau\rangle|^2
    \ll \frac{\Lambda(1/2, E_\tau\times g\times u_k)}{|\Lambda(1+2i\tau)|^2\Lambda(1, \sym^2g)\Lambda(1, \sym^2u_k)},
\end{equation}
\begin{equation}
    |\langle E_tg , u_j\rangle|^2
    \ll \frac{\Lambda(1/2, E_t\times g\times u_j)}{|\Lambda(1+2it)|^2\Lambda(1, \sym^2g)\Lambda(1, \sym^2u_j)},
\end{equation}
and
\begin{equation}\label{eqn:|<Eg, E>|^2}
    |\langle E_tg , E_{\tau}\rangle|^2
    \ll \frac{\Lambda(1/2, E_\tau \times E_t\times g)}{|\Lambda(1+2i\tau)\Lambda(1+2it)|^2\Lambda(1, \sym^2g)}.
\end{equation}
From Watson's formula \cite{watson2008rankin}, we have
\begin{equation}
    |\langle u_kg , u_j\rangle|^2\ll \frac{\Lambda(1/2, u_k\times g\times u_j)}{\Lambda(1, \sym^2u_k)\Lambda(1, \sym^2g)\Lambda(1, \sym^2u_j)},
\end{equation}
and
\begin{equation}
    |\langle u_j , f^2\rangle|^2\ll \frac{\Lambda(1/2, u_j\times \sym^2 f)\Lambda(1/2, u_j)}{\Lambda(1, \sym^2 f)^2\Lambda(1, \sym^2u_j)}.
\end{equation}
By Stirling's formula, we have
\begin{multline}\label{eqn:bound_cusp}
    |\langle u_kg , u_j\rangle \langle u_j, f^2 \rangle|\ll
    \frac{L(1/2, u_j)^{\frac{1}{2}}L(1/2, \sym^2f\times u_j)^{\frac{1}{2}}
    L(1/2, u_k\times g\times u_j)^{\frac{1}{2}}}
    {L(1, \sym^2u_j)L(1, \sym^2f)L(1, \sym^2u_k)^{\frac{1}{2}}L(1, \sym^2g)^{\frac{1}{2}}}
    \\ \times
    \frac{e^{-\frac{\pi}{2}Q_1(t_j; t_f, t_g, t_k)}}
    {t_j^{\frac{1}{2}}\prod_{\pm}(1+|t_j\pm2t_f|)^{\frac{1}{4}}\prod_{\pm, \pm}(1+|t_j\pm t_g\pm t_k|)^{\frac{1}{4}}},
\end{multline}
\begin{multline}\label{eqn:bound_est}
    |\langle E_tg , u_j\rangle \langle u_j, f^2 \rangle|\ll
    \frac{L(1/2, u_j)^{\frac{1}{2}}L(1/2, \sym^2f\times u_j)^{\frac{1}{2}}
    L(1/2, E_t\times g\times u_j)^{\frac{1}{2}}}
    {L(1, \sym^2u_j)L(1, \sym^2f)|\zeta(1+2it)|L(1, \sym^2g)^{\frac{1}{2}}}
    \\ \times
    \frac{e^{-\frac{\pi}{2}Q_1(t_j; t_f, t_g, t_k)}}
    {t_j^{\frac{1}{2}}\prod_{\pm}(1+|t_j\pm2t_f|)^{\frac{1}{4}}\prod_{\pm, \pm}(1+|t_j\pm t_g\pm t|)^{\frac{1}{4}}},
\end{multline}
\begin{multline}\label{eqn:bound_ugEEf2}
    |\langle u_kg , E_{\tau}\rangle \langle E_{\tau}, f^2\rangle|\ll
     \frac{|\zeta(1/2+i\tau)L(1/2+i\tau, \sym^2f)
    L(1/2+i\tau, u_k\times g)|}
    {|\zeta(1+2i\tau)|L(1, \sym^2f)L(1, \sym^2u_k)^{\frac{1}{2}}L(1, \sym^2g)^{\frac{1}{2}}}
    \\ \times
    \frac{e^{-\frac{\pi}{2}Q_1(\tau; t_f, t_g, t_k)}}
    {|\tau|^{\frac{1}{2}}\prod_{\pm}(1+|\tau\pm2t_f|)^{\frac{1}{4}}\prod_{\pm, \pm}(1+|\tau\pm t_g\pm t_k|)^{\frac{1}{4}}},
\end{multline}
and
\begin{multline}\label{eqn:bound_EgEEf2}
    |\langle E_tg , E_{\tau}\rangle \langle E_{\tau}, f^2\rangle|\ll
    \frac{|\zeta(1/2+i\tau)L(1/2+i\tau, \sym^2f)\prod_{\pm}
    L(1/2+i\tau\pm it, g)|}
    {|\zeta(1+2i\tau)\zeta(1+2it)|L(1, \sym^2f)L(1, \sym^2g)^{\frac{1}{2}}}
    \\ \times
    \frac{e^{-\frac{\pi}{2}Q_1(\tau; t_f, t_g, t)}}
    {|\tau|^{\frac{1}{2}}\prod_{\pm}(1+|\tau\pm2t_f|)^{\frac{1}{4}}\prod_{\pm, \pm}(1+|\tau\pm t_g\pm t|)^{\frac{1}{4}}},
\end{multline}
where
\begin{multline}\label{eqn:Q1}
Q_1(t_j; t_f, t_g, t_k)=\Big|\frac{t_j}{2}
+t_f\Big|+\Big|\frac{t_j}{2}
-t_f\Big|
+\frac{|t_j+t_g+t_k|}{2}
+\frac{|t_j+t_g-t_k|}{2}
\\
+\frac{|t_j-t_g+t_k|}{2}
+\frac{|t_j-t_g-t_k|}{2}
-t_j-2t_f-t_k-t_g.
\end{multline}
Moreover, if $2t_f\leq t_g-t_k$, we have
\begin{equation}\label{eqn:Q1_cases(tf_small)}
Q_1(t_j; t_f, t_g, t_k)=\begin{cases}
t_g-t_k-t_j,     &0\leq t_j\leq 2t_f,\\
t_g-2t_f-t_k,       &2t_f<t_j\leq t_g-t_k,\\
t_j-2t_f, &t_g-t_k<t_j\leq t_g+t_k,\\
2t_j-2t_f-t_g-t_k,  &t_g+t_k<t_j,\\
\end{cases}
\end{equation}
if $t_g-t_k<2t_f\leq t_g+t_k$, we have
\begin{equation}\label{eqn:Q1_cases(tf_middle)}
Q_1(t_j; t_f, t_g, t_k)=\begin{cases}
t_g-t_k-t_j,     &0\leq t_j\leq t_g-t_k,\\
0,       &t_g-t_k<t_j\leq 2t_f,\\
t_j-2t_f, &2t_f<t_j\leq t_g+t_k,\\
2t_j-2t_f-t_g-t_k,  &t_g+t_k<t_j,\\
\end{cases}
\end{equation}
and if $t_g+t_k<2t_f$, we have
\begin{equation}\label{eqn:Q1_cases(tf_large)}
Q_1(t_j; t_f, t_g, t_k)=\begin{cases}
t_g-t_k-t_j,     &0\leq t_j\leq t_g-t_k,\\
0,       &t_g-t_k<t_j\leq t_g+t_k,\\
t_j-t_g-t_k, &t_g+t_k<t_j\leq 2t_f,\\
2t_j-2t_f-t_g-t_k,  &2t_f<t_j.\\
\end{cases}
\end{equation}
\subsection{Proof of Lemma \ref{lemma:ef^2g}}
From Parseval's identity, we have
\begin{equation}\label{eqn:Parseval_ugf^2}
    \langle u_kg , f^2\rangle
=
\frac{3}{\pi}\langle u_kg , 1\rangle\langle 1, f^2\rangle
+
\sum_{j\geq 1}\langle u_kg , u_j\rangle \langle u_j, f^2 \rangle
+
\frac{1}{4\pi}\int_{\mathbb{R}}\langle u_kg , E_{\tau}\rangle \langle E_{\tau}, f^2\rangle \dd \tau,
\end{equation}
and
\begin{equation}\label{eqn:Parseval_Egf^2}
     \langle E_tg , f^2\rangle
=
\frac{3}{\pi}\langle E_tg , 1\rangle\langle 1, f^2\rangle
+
\sum_{j\geq 1}\langle E_tg , u_j\rangle \langle u_j, f^2 \rangle
+
\frac{1}{4\pi}\int_{\mathbb{R}}\langle E_tg , E_{\tau}\rangle \langle E_{\tau}, f^2\rangle \dd \tau.
\end{equation}
\subsubsection{The constant contribution}
The constant function contributes to $\langle u_kg,f^2\rangle$ is
\begin{equation}\label{eqn:const_u}
  \frac{3}{\pi}\langle u_kg, 1\rangle\langle 1, f^2\rangle
  =\langle u_kg, 1\rangle,
\end{equation}
and the constant function contributes to $\langle E_tg,f^2\rangle$ is
\begin{equation}\label{eqn:const_e}
  \frac{3}{\pi}\langle E_tg, 1\rangle\langle 1, f^2\rangle
  = 0.
\end{equation}
\subsubsection{The Eisenstein series contribution and the cusp form contribution}
We will show the details of estimating the cusp form contribution $\sum_{j\geq 1}\langle u_kg , u_j\rangle \langle u_j, f^2 \rangle$ in \eqref{eqn:Parseval_ugf^2}. The method for treating
$\int_{\mathbb{R}}\langle u_kg , E_{\tau}\rangle \langle E_{\tau}, f^2\rangle \dd \tau$,
$\sum_{j\geq 1}\langle E_tg , u_j\rangle \langle u_j, f^2 \rangle$
and $\int_{\mathbb{R}}\langle E_tg , E_{\tau}\rangle \langle E_{\tau}, f^2\rangle \dd \tau$
is similar.\par
Note that $t_k\ll \max(t_f, t_g)^{o(1)}$.
If $\log t_f=o(\log t_g)$, by \eqref{eqn:Q1_cases(tf_small)}, we have $Q_1(t_j; t_f, t_g, t_k)\geq t_g-t_g^{o(1)}$.
Using the convexity bound for $L$-functions, we have
\begin{multline}\label{eqn:tfverysmall}
\sum_{j\geq 1}|\langle u_kg , u_j\rangle \langle u_j, f^2 \rangle|\\
\ll
    \sum_{t_j\leq  t_g+t_g^{o(1)}}
    \frac{t_g^{O(1)}\exp(-\frac{\pi}{2}(t_g-t_g^{o(1)}))}
    {t_j^{\frac{1}{2}}\prod_{\pm}(1+|t_j\pm2t_f|)^{\frac{1}{4}}\prod_{\pm, \pm}(1+|t_j\pm t_g\pm t_k|)^{\frac{1}{4}}}
    +O(e^{-\frac{\pi}{2}t_g+t_g^{o(1)}})
    \ll e^{-\frac{\pi}{2}t_g+t_g^{o(1)}}.
\end{multline}
If $\log t_g=o(\log t_f)$, by \eqref{eqn:Q1_cases(tf_large)}, we have $Q_1(t_j; t_f, t_g, t_k)\geq t_f^{o(1)}$ for $t_j\geq t_g+2t_f^{o(1)}$ which contributes a negligibly small error term in $\sum_{j\geq 1}\langle u_kg , u_j\rangle \langle u_j, f^2 \rangle$.
Using GLH for $L$-functions, we have
\begin{multline}\label{eqn:tfverylarge}
\sum_{j\geq 1}|\langle u_kg , u_j\rangle \langle u_j, f^2 \rangle|\\
\ll
    \sum_{t_j\leq  t_g+2t_f^{o(1)}}
    \frac{t_f^{o(1)}}
    {t_j^{\frac{1}{2}}\prod_{\pm}(1+|t_j\pm2t_f|)^{\frac{1}{4}}\prod_{\pm, \pm}(1+|t_j\pm t_g\pm t_k|)^{\frac{1}{4}}}
    \ll t_f^{-\frac{1}{2}+o(1)}.
\end{multline}
\par
Now we consider the case of $\log t_f\asymp \log t_g$.
We first deal with the case $2t_f\leq t_g-t_g^\varepsilon$.
Since $t_k\ll t_g^\varepsilon$, we have  $2t_f\leq t_g-t_k$.
By \eqref{eqn:Q1_cases(tf_small)}, one can truncate the sum for $t_j$ in $[2t_f-t_g^{o(1)}, t_g+t_g^{o(1)}]$.
Using the convexity bound for $L$-functions,
 we have
\begin{equation*}
\begin{split}
\sum_{j\geq 1}&|\langle u_kg , u_j\rangle \langle u_j, f^2 \rangle|\\
    &\ll
    \sum_{2t_f-t_g^{o(1)}\leq t_j\leq  t_g+t_g^{o(1)}}
    \frac{t_g^{O(1)}\exp(-\frac{\pi}{2}(t_g-2t_f-t_k))}
    {t_j^{\frac{1}{2}}\prod_{\pm}(1+|t_j\pm2t_f|)^{\frac{1}{4}}\prod_{\pm, \pm}(1+|t_j\pm t_g\pm t_k|)^{\frac{1}{4}}}
    \\
    &\qquad \qquad  \qquad \qquad
    +O(e^{-\frac{\pi}{2}(t_g-2t_f-t_k
    +t_g^{o(1)})})\\
    &
    \ll e^{-\frac{\pi}{2}(t_g-2t_f)+t_g^{o(1)}}.
    \end{split}
\end{equation*}
Now we deal with the case $ t_g-t_g^\varepsilon\leq 2t_f$.
From now on, we assume GLH in the rest of this section.
When $t_g-t_g^\varepsilon <2t_f\leq t_g-t_k$,
from \eqref{eqn:Q1_cases(tf_small)},
we have $Q_1(t_j; t_f, t_g, t_k)\geq t_g^{o(1)}$ for $t_j\notin [t_g-t_g^{o(1)},  2t_f+t_g^{o(1)}]$
which contributes a negligibly small error term in $\sum_{j\geq 1}\langle u_kg , u_j\rangle \langle u_j, f^2 \rangle$.
Thus we have
\begin{multline*}
\sum_{j\geq 1}|\langle u_kg , u_j\rangle \langle u_j, f^2 \rangle|
\\
\ll
   \sum_{2t_f-t_g^{o(1)}\leq t_j\leq  t_g+t_g^{o(1)}}
    \frac{t_g^{o(1)}}
    {t_j^{\frac{1}{2}}\prod_{\pm}(1+|t_j\pm2t_f|)^{\frac{1}{4}}\prod_{\pm, \pm}(1+|t_j\pm t_g\pm t_k|)^{\frac{1}{4}}}\\
    \ll \sum_{2t_f-t_g^{o(1)}\leq t_j\leq  t_g+t_g^{o(1)}}
    \frac{t_g^{-\frac{3}{4}+o(1)}}
    {t_j^{\frac{1}{2}}}
    \ll t_g^{-\frac{1}{4}+\varepsilon+o(1)} (\asymp t_f^{-\frac{1}{4}+o(1)}{t_g}^{\varepsilon}).
\end{multline*}
When $t_g-t_k <2t_f\leq t_g+t_k$,
from \eqref{eqn:Q1_cases(tf_middle)},
we have $Q_1(t_j; t_f, t_g, t_k)\geq t_g^{o(1)}$ for $t_j\notin [2t_f-t_g^{o(1)},  t_g+t_g^{o(1)}]$
which contributes a negligibly small error term in $\sum_{j\geq 1}\langle u_kg , u_j\rangle \langle u_j, f^2 \rangle$.
Thus we have
\begin{multline*}
\sum_{j\geq 1}|\langle u_kg , u_j\rangle \langle u_j, f^2 \rangle|
\ll
   \sum_{2t_f-t_g^{o(1)}\leq t_j\leq  t_g+t_g^{o(1)}}
    \frac{t_g^{o(1)}}
    {t_j^{\frac{1}{2}}\prod_{\pm}(1+|t_j\pm2t_f|)^{\frac{1}{4}}\prod_{\pm, \pm}(1+|t_j\pm t_g\pm t_k|)^{\frac{1}{4}}}
    \\
    \ll t_g^{-\frac{1}{4}+o(1)} (\asymp t_f^{-\frac{1}{4}+o(1)}).
\end{multline*}
When $t_g+t_k<2t_f$,
from \eqref{eqn:Q1_cases(tf_large)},
we have $Q_1(t_j; t_f, t_g, t_k)\geq t_f^{o(1)}$ for $t_j\notin [t_g-t_f^{o(1)},  t_g+t_f^{o(1)}]$
which contributes a negligibly small error term.
Thus we have
\begin{multline}\label{eqn:tgtkleq2tf}
\sum_{j\geq 1}|\langle u_kg , u_j\rangle \langle u_j, f^2 \rangle|\\
\ll
    \sum_{t_g-t_f^{o(1)}\leq t_j\leq  t_g+t_f^{o(1)}}
    \frac{t_f^{o(1)}}
    {t_j^{\frac{1}{2}}\prod_{\pm}
    (1+|t_j\pm2t_f|)^{\frac{1}{4}}\prod_{\pm, \pm}(1+|t_j\pm t_g\pm t_k|)^{\frac{1}{4}}}\\
    \ll t_f^{-\frac{1}{4}+o(1)}(2t_f-t_g)^{-\frac{1}{4}}.
\end{multline}
Combining the bounds in the above cases, we have
\begin{equation}
\sum_{j\geq 1}\langle u_kg , u_j\rangle \langle u_j, f^2 \rangle
\ll \begin{cases}
t_f^{-\frac{1}{4}+o(1)} (1+|2t_f-t_g|)^{-\frac{1}{4}} , &\textrm{if } t_g-t_g^{\varepsilon}\leq 2t_f,\\
e^{-\frac{\pi}{2}(t_g-2t_f)+t_g^{o(1)}},     &\textrm{if } 2t_f\leq t_g-t_g^{\varepsilon}.\\
\end{cases}
\end{equation}
Similarly,
we have the same bounds for
$\int_{\mathbb{R}}\langle u_kg , E_{\tau}\rangle \langle E_{\tau}, f^2\rangle \dd \tau$,
$\sum_{j\geq 1}\langle E_tg , u_j\rangle \langle u_j, f^2 \rangle$
and $\int_{\mathbb{R}}\langle E_tg , E_{\tau}\rangle \langle E_{\tau}, f^2\rangle \dd \tau$.
Combining these bounds with \eqref{eqn:const_e}, \eqref{eqn:const_u} and Parseval's identities \eqref{eqn:Parseval_ugf^2}, \eqref{eqn:Parseval_Egf^2}, we finish the proof of Lemma \ref{lemma:ef^2g}.

\section{Proof of Theorem \ref{thm:22}}\label{sec:22}



In this section, we will prove Theorem \ref{thm:22} by assuming Proposition \ref{prop:mixedmonents1} below.
From Parseval's identity, we have
\begin{equation}\label{eqn:Parseval's identity}
\langle f^2 , g^2\rangle
=
\frac{3}{\pi}\langle f^2 , 1\rangle\langle 1, g^2\rangle
+
\sum_{j\geq 1}\langle f^2 , u_j\rangle \langle u_j, g^2 \rangle
+
\frac{1}{4\pi}\int_{\mathbb{R}}\langle f^2 , E_t\rangle \langle E_t, g^2\rangle \dd t.
\end{equation}
By our normalization \eqref{eqn:L2normalization}, we know the constant function contribution to $\frac{\langle f^2,g^2\rangle}{\vol(\Gamma\backslash \mathbb{H})}$ is
\begin{equation}\label{eqn:constant}
  \frac{9}{\pi^2}\langle f^2 , 1\rangle\langle 1, g^2\rangle
  = 1.
\end{equation}
Here we know $\vol(\Gamma\backslash \mathbb{H})=\pi/3$.
Now we estimate
\begin{equation*}\label{}
\int_{\mathbb{R}}\langle f^2 , E_t\rangle \langle E_t, g^2\rangle \dd t.
\end{equation*}
By \cite[Lemma 2.1]{ButtcaneKhan2017}, we have
\begin{equation*}
\int_{\mathbb{R}}|\langle f^2 , E_t\rangle |^2\dd t\ll t_f^{-1+\varepsilon},
\end{equation*}
under GLH.
Applying the Cauchy--Schwarz inequality, we have
\begin{equation}\label{eqn:boundforEisenstein}
\int_{\mathbb{R}}\langle f^2 , E_t\rangle \langle E_t, g^2\rangle \dd t\leq \left(\int_{\mathbb{R}}|\langle f^2 , E_t\rangle |^2\dd t \right)^{\frac{1}{2}}\left(\int_{\mathbb{R}}|\langle g^2 , E_t\rangle |^2\dd t \right)^{\frac{1}{2}}\ll (t_ft_g)^{-\frac{1}{2}+\varepsilon}.
\end{equation}
Then we need to show that the cusp form contribution $\sum_{j\geq 1}\langle f^2 , u_j\rangle \langle u_j, g^2 \rangle$ is small.
We only need to deal with the case when $u_j$ is even, otherwise we have $\langle f^2 , u_j\rangle=0$.
Watson's formula \cite{watson2008rankin} gives
\begin{equation*}
|\langle f^2 , u_j\rangle|^2\ll \frac{\Lambda(1/2, u_j)\Lambda(1/2, \sym^2f \times u_j)}{\Lambda(1, \sym^2f)^2\Lambda(1, \sym^2u_j)}.
\end{equation*}
Hence we have
\begin{multline}\label{eqn:cusp_part}
\sum_{j\geq 1}\langle f^2 , u_j\rangle \langle u_j, g^2 \rangle
\leq \sum_{j\geq 1} |\langle f^2 , u_j\rangle \langle u_j, g^2 \rangle|
\\
\ll \sum_{j\geq 1}\frac{L(1/2, u_j)L(1/2, \sym^2f\times u_j)^{\frac{1}{2}}L(1/2, \sym^2g \times u_j)^{\frac{1}{2}}}{L(1,\sym^2f)L(1,\sym^2g)
L(1,\sym^2u_j)}H(t_j; t_f, t_g),
\end{multline}
where
\begin{equation}\label{weight H}
H(t_j; t_f, t_g)=
\frac{L_{\infty}(1/2, u_j)L_{\infty}(1/2, \sym^2f\times u_j)^{\frac{1}{2}}L_{\infty}(1/2, \sym^2g \times u_j)^{\frac{1}{2}}}
{L_{\infty}(1,\sym^2f)L_{\infty}(1,\sym^2g)L_{\infty}(1,\sym^2u_j)}.
\end{equation}
By Stirling's formula, one derives
\begin{equation}\label{}
H(t_j; t_f, t_g)\ll
\frac{\exp(-\frac{\pi}{2}Q(t_j; t_f, t_g))}
{|t_j|\prod_{\pm}(1+|t_j\pm 2t_f|)^{\frac{1}{4}}\prod_{\pm}(1+|t_j\pm 2t_g|)^{\frac{1}{4}}},
\end{equation}
where
\begin{equation}\label{eqn:Q(tj; tf,tg)}
Q(t_j; t_f, t_g)=|t_f+\frac{t_j}{2}|+|t_f-\frac{t_j}{2}|
+|t_g+\frac{t_j}{2}|+|t_g-\frac{t_j}{2}|
-2t_f-2t_g.
\end{equation}
More specifically, assume $t_f\leq t_g$, we have
\begin{equation}\label{eqn:Q_cases}
Q(t_j; t_f, t_g)=\begin{cases}
0,              &0\leq t_j\leq 2t_f,\\
t_j-2t_f,       &2t_f<t_j\leq 2t_g,\\
2t_j-2t_f-2t_g, &2t_g<t_j.
\end{cases}
\end{equation}
Note that the contribution of $t_j > 2t_f+{t^{\varepsilon}_g}$ is negligibly small due to the rapidly decay of the exponential function.
Then (\ref{eqn:cusp_part}) is reduced to bounding the following mixed moment
\begin{equation}\label{mixed moment}
\frac{{(\log\log t_g)}^{3}}
{(t_f t_g)^{\frac{1}{4}}}
\sum_{t_j\leq 2t_f+{t^{\varepsilon}_g}}
\frac{L(1/2, u_j)
L(1/2, \sym^2f\times u_j)^{\frac{1}{2}}
L(1/2, \sym^2g \times u_j)^{\frac{1}{2}}}
{|t_j|(1+|t_j-2t_f|)^{\frac{1}{4}}(1+|t_j- 2t_g|)^{\frac{1}{4}}}.
\end{equation}
Under GRH and GRC, by \cite[Theorem 1.1]{WX16}, we have
\begin{equation}
  \frac{1}{\log\log t_f} \ll L(1,\sym^2 f) \ll (\log\log t_f)^3.
\end{equation}

We consider three cases:
\begin{itemize}
  \item[a)] $t_f < t_g^{\varepsilon}$,
  \item[b)] $t_g^{\varepsilon}\leq t_f \leq t_g^{1-5\varepsilon}$,
  \item[c)] $t_g^{1-5\varepsilon}\leq t_f \leq t_g$.
\end{itemize}
Assume GRH. Then we have the following bounds 
\begin{equation}\label{eqn:GLH}
  \begin{split}
    &L(1/2, u_j) \ll t_j^\varepsilon, \\
    &L(1/2, \sym^2f\times u_j) \ll (t_f+t_j)^\varepsilon,  \\
    &L(1/2, \sym^2g \times u_j) \ll (t_g+t_j)^\varepsilon.
  \end{split}
\end{equation}
In case a),  \eqref{mixed moment} is bounded by
\begin{equation}\label{eqn:22 small}
   O\Big( t_g^{-\frac{1}{4}+\varepsilon}\sum_{t_j\leq 3t_g^{\varepsilon}} t_j^{-1}
  (1+t_j+t_g)^{-\frac{1}{4}} \Big)
  = O\left( t_g^{-\frac{1}{2}+2\varepsilon}\right).
\end{equation}
In case b), we denote $t_f=t_g^{\delta}$ where $\varepsilon\leq \delta\leq 1-5\varepsilon$, then the contribution to \eqref{mixed moment} is
\begin{equation*}
  \ll t_g^{-\frac{1}{2}-\frac{\delta}{4}
  +\varepsilon}
  \Big(\sum_{t_j\leq 2t_f-t_f^{\varepsilon}} t_j^{-1}(1+|t_j-2t_f|)^{-\frac{1}{4}}
  +\sum_{2t_f-t_f^{\varepsilon}
  \leq t_j \leq 2t_f+t_f^{\varepsilon}}
  t_j^{-1}
  \Big)
  \ll t_g^{-\frac{1-\delta}{2}
  +2\varepsilon}.
\end{equation*}
Here we have used the fact $\sum_{t_j \in (T,T+M)} 1 \ll TM$ for any $M\in [1,T]$.
In case c), we need the following proposition, which will be proven in the next section.

\begin{proposition}\label{prop:mixedmonents1}
  Assume GRH and GRC. Let $f,g$ be two orthogonal Hecke--Maass cusp forms over $\SL_2(\mathbb{Z})$ with spectral parameters $t_f, t_g$ with $t_g^{\delta}\leq t_f\leq t_g$ for some $\delta>0$. Let $\varepsilon>0$ be a sufficiently small number, $t_f^{1-\varepsilon}\leq X\leq 3t_g$, and $X^{\varepsilon}\leq Y \leq X$. 
  Then for any positive real numbers $\ell_1,\ell_2,\ell_3$ we have
  \begin{multline}\label{eqn:mixedmoments}
    \sum_{X<t_j\leq X+Y}
    L(1/2, u_j)^{\ell_1}
    L(1/2, \sym^2f\times u_j)^{\ell_2}
    L(1/2, \sym^2g \times u_j)^{\ell_3}
    \\ \ll_\varepsilon XY(\log X)
    ^{\frac{\ell_1(\ell_1-1)}{2}
    +\frac{\ell_2(\ell_2-1)}{2}
    +\frac{\ell_3(\ell_3-1)}{2}+\varepsilon}.
  \end{multline}
\end{proposition}


Now we can bound the mixed moment (\ref{mixed moment}) which gives the bound for the cusp form contribution
$\sum_{j\geq 1}\langle f^2 , u_j\rangle \langle u_j, g^2 \rangle$.
We split the summand in (\ref{mixed moment}) into four parts,
\begin{multline}
\frac{{(\log\log t_f)}^{3}}
{(t_f t_g)^{\frac{1}{4}}}
\left\{\sum_{t_j\leq {t_f}^{1-\varepsilon}}
+\sum_{{t_f}^{1-\varepsilon}< t_j\leq t_f}
+\sum_{t_f< t_j \leq 2t_f-{t_f}^{\varepsilon}}
+\sum_{2t_f-{t_f}^{\varepsilon}<t_j \leq 2t_f+{t^{\varepsilon}_f}}
\right\}\\
\frac{L(1/2, u_j)
L(1/2, \sym^2f\times u_j)^{\frac{1}{2}}
L(1/2, \sym^2g \times u_j)^{\frac{1}{2}}}
{|t_j|(1+|t_j-2t_f|)^{\frac{1}{4}}(1+|t_j- 2t_g|)^{\frac{1}{4}}}.
\end{multline}

In the range of $t_j\leq {t_f}^{1-\varepsilon}$, using \eqref{eqn:GLH}, we get that the contribution is bounded by
$O({t_f}^{-\frac{1}{2}+\varepsilon} {t_g}^{-\frac{1}{2}})$.\par

In the range of ${t_f}^{1-\varepsilon}< t_j\leq t_f$, we have $|t_j-2t_f|\gg t_f$, $ |t_j-2t_g|\gg t_g$, the summand contributes
\begin{equation}
\frac{{(\log\log t_f)}^{3}}
{(t_f t_g)^{\frac{1}{2}}}
\sum_{{t_f}^{1-\varepsilon}< t_j\leq t_f}
\frac{1}{t_j}L(1/2, u_j)
L(1/2, \sym^2f\times u_j)^{\frac{1}{2}}
L(1/2, \sym^2g \times u_j)^{\frac{1}{2}}.
\end{equation}
Applying the partial summation and Proposition \ref{prop:mixedmonents1}, we get the bound $O\left(\frac{{t_f}^{\frac{1}{2}}}{{t_g}^{\frac{1}{2}}}(\log t_f)^{-\frac{1}{4}+\varepsilon}\right)$.\par

In the range of $t_f< t_j \leq 2t_f-{t_f}^{\varepsilon}$ with any fixed $\varepsilon>0$, we have $t_j\asymp t_f$ and  ${t_f}^{\varepsilon}\leq|t_j-2t_f|<t_f$. Dividing the summation into small intervals and using the Proposition \ref{prop:mixedmonents1}, the summand contributes
\begin{multline}
    \frac{(\log\log t_f)^3}{{t_f}^\frac{5}{4}{t_g}^\frac{1}{4}}
    \sum_{t_f< t_j\leq 2t_f-{t_f}^{\varepsilon}}\frac{L(1/2, u_j)
    L(1/2, \sym^2f\times u_j)^{\frac{1}{2}}
    L(1/2, \sym^2g \times u_j)^{\frac{1}{2}}}
    {(1+2t_f-t_j)^\frac{1}{4}
    (1+2t_g-2t_f+2t_f-t_j)^\frac{1}{4}}\\
    \ll
    \frac{(\log\log t_f)^3}{{t_f}^\frac{5}{4}{t_g}^\frac{1}{4}}
    \sum_{\substack{k< t_f^{1-\varepsilon}\\
    (k-1)t_f^{\varepsilon}< 2t_f-t_j\leq kt_f^{\varepsilon}}}
    \frac{
    L(1/2, u_j)
    L(1/2, \sym^2f\times u_j)^{\frac{1}{2}}
    L(1/2, \sym^2g \times u_j)^{\frac{1}{2}}}{(1+kt_f^\varepsilon)^{\frac{1}{4}}(1+2t_g-2t_f+kt_f^\varepsilon)^{\frac{1}{4}}}\\
    \\
    \ll \frac{(\log\log t_f)^3}{{t_f}^\frac{5}{4}{t_g}^\frac{1}{4}}
    \sum_{k<t_f^{1-\varepsilon}}
    \frac{t_f^{1+\varepsilon}(\log t_f)^{-\frac{1}{4}+\varepsilon}}
    {(1+kt_f^\varepsilon)^{\frac{1}{4}}
    (1+2t_g-2t_f+kt_f^\varepsilon)^{\frac{1}{4}}}
    \\ \ll \frac{{t_f}^{\frac{1}{2}}}{\max(t_g-t_f, t_f)^{\frac{1}{4}}{t_g}^{\frac{1}{4}}}{(\log t_f)}^{-\frac{1}{4}+\varepsilon}.
\end{multline}
Note that $\max(t_g-t_f, t_f)\asymp t_g$, thus it is bounded by
$O\left(\frac{{t_f}^{\frac{1}{2}}}{{t_g}^{\frac{1}{2}}}{(\log t_f)}^{-\frac{1}{4}+\varepsilon}\right).$\par

In the range of $2t_f-{t_f}^{\varepsilon}<t_j \leq 2t_f+{t^{\varepsilon}_f}$, using \eqref{eqn:GLH}, the summation is bounded by
\begin{multline}
    \frac{(\log\log t_f)^3}{{t_f}^\frac{5}{4}{t_g}^\frac{1}{4}}\sum_{ 2t_f-t_f^{\varepsilon}<t_j \leq 2t_f+{t^{\varepsilon}_f}}
    L(1/2, u_j)
    L(1/2, \sym^2f\times u_j)^{\frac{1}{2}}
    L(1/2, \sym^2g \times u_j)^{\frac{1}{2}}
    \\ \ll {t_f}^{-\frac{1}{4}+\varepsilon}{t_g}^{-\frac{1}{4}}.
\end{multline}
\par

Combining the above estimates,  we get that (\ref{mixed moment}) is bounded by
$$O\left(\frac{{t_f}^{\frac{1}{2}}}{{t_g}^{\frac{1}{2}}}{(\log t_f)}^{-\frac{1}{4}+\varepsilon}\right).$$
Therefore
\begin{equation}\label{eqn:boundforcuspform}
    \sum_{j\geq 1}\langle f^2 , u_j\rangle \langle u_j, g^2 \rangle\ll
    \frac{{t_f}^{\frac{1}{2}}}{{t_g}^{\frac{1}{2}}}{(\log t_f)}^{-\frac{1}{4}+\varepsilon}.
\end{equation}
Theorem \ref{thm:22} follows upon combining \eqref{eqn:Parseval's identity}, \eqref{eqn:constant}, \eqref{eqn:boundforEisenstein} and \eqref{eqn:boundforcuspform}.

\section{Upper bounds for mixed moments of $L$-functions}\label{sec:fractional-moment}

In this section, we will prove Proposition \ref{prop:mixedmonents1}, by using Soundararajan's method \cite{soundararajan2009moments}. See also \cite{HL23}.

\subsection{Average of $\lambda_{u_j}(n)$ on the spectral aspect}

Let $\lambda_{u_j}(n)$ denote the eigenvalue of the $n$th Hecke operator corresponding to $u_j$.  In order to bound the mixed moment \eqref{mixed moment} above via Soundararajan's method, we require the average of $\lambda_{u_j}(n)$ on the spectral aspect.
Our tool for summing over the spectrum is the Kuznetsov trace formula as follow.
Let
\[
  S(a,b;c) = \sideset{}{^*}\sum_{d \bmod c} e\left(\frac{ad+b\bar{d}}{c}\right)
\]
be the classical Kloosterman sum. For any $n, m \geq1$, and any test function
$h(t)$ which is even and satisfies the following conditions:
\begin{itemize}
  \item [(i)] $h(t)$ is holomorphic in $|\Im(t)|\leq 1/2+\varepsilon$,
  \item [(ii)] $h(t)\ll (1+|t|)^{-2-\varepsilon}$ in the above strip.
\end{itemize}
Then we have the following Kuznetsov trace formula:
 \begin{multline}\label{eqn:KTF}
         \sum_{j\geq 1}w_jh(t_j)\lambda_{u_j}(n)\lambda_{u_j}(m)+\int_{-\infty}^{\infty}w(t)h(t)\eta_t(n)\eta_t(m)\,\dd t\\
         =\frac{\delta(n,m)}{\pi}\int_{-\infty}^{\infty}h(t)\tanh(\pi t)t\,\dd t+\sum_{c\geq 1}\frac{S(n, m;c)}{c}\mathcal{J}\left(\frac{4\pi \sqrt{nm}}{c}\right).
 \end{multline}
where $\delta(n, m)$  is the Kronecker symbol and
\[
w_j=\frac{2\pi}{L(1, \sym^2 u_j)},  \quad w(t)=\frac{1}{|\zeta(1+2i t)|^2},
\]
\[
\eta_t(n)=\sum_{ad=n}\left(\frac{a}{d}\right)^{2it},
\]
\[
\mathcal{J}(x)=2i\int_{-\infty}^{\infty}J_{2it}\frac{h(t)t}{\cosh(\pi t)}\dd t.
\]
 and $J_\nu(x)$ is the standard $J$-Bessel function.

Similarly to the choice of weight function in \cite{IvicJutila03},
let $X$ be a large real number, for $Y \leq X$ and $1 \leq M\leq Y/\log X$,
we define the weight functions
    \[
    h(t, T, M):=e^{-\left(\frac{t-T}{M}\right)^2}
    +e^{-\left(\frac{t+T}{M}\right)^2},
    \]
and
    \[
    \Phi_{X,Y,M}(t):=\frac{1}{\sqrt{\pi}M}\int_{X}^{X+Y}h(t, T, M)\dd T.
    \]
We know that for any fixed $A>0$, there exists $c>0$ such that (see \cite[(2.14) and (3.2)]{IvicJutila03} or \cite[Section 3]{BHS2021})
\begin{equation}
    \Phi_{X,Y,M}(t)=
    \begin{cases}
    1+O(t^{-2})    &
    X+cM\sqrt{\log X}<t\\
    &\text{and } t<X+Y-cM\sqrt{\log X},\\
    O((|t|+X)^{-A})   &
    t<X-cM\sqrt{\log X} \;
    \\
    &\text{or}\; t>X+Y+cM\sqrt{\log X},\\
    1+O\Big(\frac{M^3}
    {(M+\min(|t-X|,|t-X-Y|))^{3}}\Big)  &\text{otherwise.}
    \end{cases}
\end{equation}

\begin{lemma}\label{lemma:3.1}
    For any positive integer $n\ll M^2X^{2-\varepsilon}$, we have
    \begin{equation}
    \sum_{j\geq 1}w_j\lambda_{u_j}(n)\Phi_{X,Y,M}(t_j)=
    \delta(n,1)G(X,Y,M)+O(X^{\varepsilon}Y),
    \end{equation}
    where $w_j=\frac{2\pi}{L(1,\sym^2 u_j)}$, and $G(X,Y,M)=\frac{2}{\pi}XY+\frac{1}{\pi}Y^2
    +O(MY)$.
\end{lemma}
\begin{proof}
The left hand side is equal to
\begin{equation*}
    \frac{1}{\sqrt{\pi}M}\int_{X}^{X+Y}\sum_{j\geq 1}w_j\lambda_{u_j}(n)h(t_j, T, M)\dd T.
\end{equation*}
 Applying the Kuznetsov trace formula (\ref{eqn:KTF}) with the weight function $h(t, T, M)$, we have
    \begin{equation}
    \begin{split}
           \sum_{j\geq 1}w_j\lambda_{u_j}(n)&h(t_j, T ,M)+\int_{-\infty}^{\infty}h(t, T ,M)w(t)\eta_{t}(n) \dd t\\
           &=\delta(n,1)\frac{1}{\pi}
           \int_{-\infty}^{\infty}h(t, T ,M)\tanh(\pi t)t \dd t+\sum_{c\geq 1}\frac{S(n, 1; c)}{c}
           \mathcal{J}\left(\frac{4\pi\sqrt{n}}{c}\right),
    \end{split}
    \end{equation}
    where
    \begin{equation}
      \mathcal{J}(x)= 2i \int_{-\infty}^{\infty} J_{2it}(x)\frac{h(t, T ,M)t}{\cosh(\pi t)} \dd t.
    \end{equation}
Recall that $w(t)=\frac{1}{|\zeta(1+2it)|^2}\ll (1+|t|)^{\varepsilon}$, the integration of continuous spectrum contributes
\[
     \frac{1}{\sqrt{\pi}M}\int_{X}^{X+Y}\int_{-\infty}^{\infty}h(t, T ,M)w(t)\eta_{t}(n) \dd t\,\dd T
     \ll X^{\varepsilon}\int_{|t-X|\ll Y}\Phi_{X,Y,M}(t)\dd t+O(X^{-A})
     \ll X^{\varepsilon}Y.
\]
By the uniform bound $|\mathcal{J}(y)|\ll T y^{3/4}$ and the bound $|\mathcal{J}(y)|\ll T^{-A}$ for $y\ll MT^{1-\varepsilon}\asymp MX^{1-\varepsilon}$ and any large $A>0$ \cite[Lemma 7.1]{Young2017weyl}, in addition, the Weil bound of the Kloosterman sum
$S(n, 1; c)\leq\tau(c)c^{\frac{1}{2}}$,
we get that the off-diagonal term is bounded by
\[
\sum_{c\leq T^{A/2}}+\sum_{c> T^{A/2}}\frac{1}{c^{\frac{1}{2}-\varepsilon}}
\left|\mathcal{J}\left(\frac{4\pi\sqrt{n}}{c}\right)
\right|\ll\sum_{c\leq T^{A/2}}\frac{1}{c^{\frac{1}{2}-\varepsilon}}
\frac{1}{T^{A}}+\sum_{c> T^{A/2}}\frac{T}{c^{\frac{1}{2}-\varepsilon}}
\left(\frac{4\pi\sqrt{n}}{c}\right)
^{\frac{3}{4}}\ll T^{-A/9}.
\]
Finally, the integration of the diagonal term contributes
\begin{equation*}
\begin{split}
      \delta(n,1)\frac{1}{\pi^{\frac{3}{2}}M} \int_{X}^{X+Y}\int_{-\infty}^{\infty}&h(t, T ,M)\tanh(\pi t)t \dd t\,\dd T\\
      &=  \delta(n,1)\frac{2}{\pi M}
      \int_{X}^{X+Y}
      (TM+O(M^2))\dd T\\
      &= \frac{\delta(n,1)}{\pi}
      \left(2XY+Y^2+O(MY)\right).
\end{split}
\end{equation*}
Combining the above estimates, we finish the proof of Lemma \ref{lemma:3.1}.
\end{proof}

\subsection{Bounds for $L$-values}

Let $\alpha_f,\beta_f$, $\alpha_g,\beta_g$, and $\alpha_j,\beta_j$ denote the Satake parameters for $f$, $g$, and $u_j$, respectively.
Let
\begin{equation}\label{def:Lambdauj}
  \Lambda_{u_j}(p^n)=\alpha_j(p)^n+\beta_j(p)^n,
\end{equation}
\begin{equation}\label{def:Lambdasym2fuj}
  \Lambda_{\sym^2f\times u_j}(p^n)
  =(\alpha_f(p)^{2n}+1+\beta_f(p)^{2n})
  (\alpha_j(p)^n+\beta_j(p)^n),
\end{equation}
and
\begin{equation}\label{def:Lambdasym2guj}
  \Lambda_{\sym^2g\times u_j}(p^n)
  =(\alpha_g(p)^{2n}+1+\beta_g(p)^{2n})
  (\alpha_j(p)^n+\beta_j(p)^n).
\end{equation}
In particular, we have
\begin{equation}\label{eqn:heckerelation}
  \Lambda_{\sym^2f\times u_j}(p^2)
  =(\lambda_{\sym^4f}(p)-\lambda_{\sym^2f}(p)+1)
  (\lambda_{\sym^2 u_j}(p)-1),
\end{equation}
and a similar result for $\Lambda_{\sym^2g\times u_j}(p^2)$.

\begin{lemma}\label{lemma:LogLfunctions}
  Assume GRH. Let $f,g\neq u_j$ be Hecke--Maass cusp forms over $\SL_2(\mathbb{Z})$ with spectral parameter $X\leq t_j\leq 2X$. Then for $x>10$, we have
    \begin{equation}\label{eqn:logLfunctions2}
    \log  L\Big(\frac{1}{2},u_j\Big)\leq
    \sum_{p^n\leq x}\frac{\Lambda_{u_j}(p^n)}
    {np^{n(\frac{1}{2}+\frac{1}{\log x})}}
    \frac{\log \frac{x}{p^n}}{\log x}
    +O\left(\frac{\log X}{\log x}+1\right),
  \end{equation}
  with an absolute implied constant, and
  \begin{equation}\label{eqn:logLfunctions2a}
    \log  L\Big(\frac{1}{2},\sym^2f\times u_j\Big)\leq
    \sum_{p^n\leq x}\frac{\Lambda_{\sym^2f\times u_j}(p^n)}
    {np^{n(\frac{1}{2}+\frac{1}{\log x})}}
    \frac{\log \frac{x}{p^n}}{\log x}
    +O\left(\frac{\log (X+t_f)}{\log x}+1\right),
  \end{equation}
    \begin{equation}\label{eqn:logLfunctions2b}
    \log  L\Big(\frac{1}{2},\sym^2 g\times u_j\Big)\leq
    \sum_{p^n\leq x}\frac{\Lambda_{\sym^2 g\times u_j}(p^n)}
    {np^{n(\frac{1}{2}+\frac{1}{\log x})}}
    \frac{\log \frac{x}{p^n}}{\log x}
    +O\left(\frac{\log (X+t_g)}{\log x}+1\right).
  \end{equation}
\end{lemma}
\begin{proof}
  See \cite[Theorem 2.1]{Chandee} and \cite[Lemma 6.4]{HL23}.
\end{proof}

\begin{lemma}\label{lemma:sumofcoefficients}
  Assume GRH and GRC.
  Let $f,g$ be two distinct Hecke--Maass cusp forms over $\SL_2(\mathbb{Z})$.
  For $x\geq 2$ and $f,g\ncong u_j$ we have
\begin{equation}\label{eqn:sumofcoefficients1a}
  \sum_{p\leq x}\frac{\lambda_{\sym^4f}(p)\lambda_{\sym^2 u_j}(p)}{p}=O(\log\log\log (t_j+t_f)),
  \end{equation}
\begin{equation}\label{eqn:sumofcoefficients1b}
    \sum_{p\leq x}\frac{\lambda_{\sym^4g}(p)
    \lambda_{\sym^2 u_j}(p)}{p}=O(\log\log\log (t_j+t_g)),
  \end{equation}
\begin{equation}\label{eqn:sumofcoefficients2a}
    \sum_{p\leq x}\frac{\lambda_{\sym^2f}(p)
    \lambda_{\sym^2 u_j}(p)}{p}=O(\log\log\log (t_j+t_f)),
\end{equation}
\begin{equation}\label{eqn:sumofcoefficients2b}
    \sum_{p\leq x}\frac{\lambda_{\sym^2 g}(p)\lambda_{\sym^2 u_j}(p)}{p}=O(\log\log\log (t_j+t_g)),
\end{equation}
\begin{equation}\label{eqn:sumofcoefficients3}
    \sum_{p\leq x}\frac{\lambda_{\sym^2 f}(p)
    \lambda_{\sym^2 g}(p) }{p}=O(\log\log\log(t_f+t_g)),
\end{equation}
\begin{equation}\label{eqn:sumofcoefficients4a}
    \sum_{p\leq x}\frac{\lambda_{\sym^2 f}(p)}{p}=O(\log\log\log t_f),
\end{equation}
\begin{equation}\label{eqn:sumofcoefficients4b}
    \sum_{p\leq x}\frac{\lambda_{\sym^2 g}(p)}{p}=O(\log\log\log t_g),
\end{equation}
\begin{equation}\label{eqn:sumofcoefficients4c}
    \sum_{p\leq x}\frac{\lambda_{\sym^2 u_j}(p)}{p}=O(\log\log\log t_j).
\end{equation}
\end{lemma}
\begin{proof}
  We will only establish the first bound \eqref{eqn:sumofcoefficients1a}, since others follow from similar arguments by using some facts such as $\sym^2 f\ncong \sym^2 g$.  Indeed, from \cite{GelbartJacquet} we know $\sym^2 f$ and $\sym^2 g$ are self-dual Hecke--Maass cusp forms over $\SL_3(\mathbb{Z})$, and \cite{Ramakrishnan14} tells us $\sym^2 f\ncong \sym^2 g$.
  Note that $L(s,\sym^4f\times\sym^2u_j)$ has an analytic continuation to the complex plane and satisfies a functional equation.
  Assuming GRH for $L(s,\sym^4f\times\sym^2u_j)$, it follows that $\log L(s,\sym^4f\times\sym^2u_j)$ is analytic in the region $\Re (s)\geq \frac{1}{2}+\frac{1}{\log x}$. Moreover, by repeating a classical argument of Littlewood (see Titchmarsh \cite[(14.2.2)]{Ti86}), in this region we have
  \begin{equation}\label{eqn:log L <<}
  |\log L(s,\sym^4f\times\sym^2u_j)|
  \ll \Big(\Re(s)-\frac{1}{2}\Big)^{-1}\log (t_j+t_f+|\Im(s)|).
  \end{equation}
  By Perron's formula, we have for $x\geq 2$ that
  \begin{multline}
    \sum_{p\leq x}\frac{\lambda_{\sym^4f}(p)
    \lambda_{\sym^2 u_j}(p)}{p}\\
    =\frac{1}{2\pi i}
    \int_{1-ix\log(t_j+t_f+x)}
    ^{1+ix\log(t_j+t_f+x)}
    \log L(s+1,\sym^4f\times\sym^2u_j)
    x^s\frac{\dd s}{s}+O(1).
  \end{multline}
  Shifting contours to $\Re(s)=-\frac{1}{2}+\frac{1}{\log x}$ we collect a simple pole at $s=0$ with residue $\log L(1,\sym^4f\times\sym^2u_j)$. The upper horizonal contour is bounded by
  \begin{multline}
    \ll\frac{1}{x\log(t_f+t_j+x)}
    \int_{-\frac{1}{2}+\frac{1}{\log x}
    +ix\log(t_j+t_f+x)}^{1+ix\log(t_j+t_f+x)}
    |\log L(s+1,\sym^4f\times\sym^2u_j)|
    |x^s| |\dd s|\\
    \ll \frac{\log x\log(t_f+t_j+x\log(t_j+t_f+x))}
    {x\log(t_j+t_f+x)}
    \int_{-\frac{1}{2}}^{1}x^u \dd u\ll 1,
  \end{multline}
and the lower horizontal contour is also $O(1)$.
Hence by \eqref{eqn:log L <<} we have for $x\geq 2$ that
\begin{multline}
   \sum_{p\leq x}\frac{\lambda_{\sym^4f}(p)
    \lambda_{\sym^2 u_j}(p)}{p}\\
    =\log L(1,\sym^4f\times\sym^2u_j)
    +O\Big(\frac{\log x}{\sqrt{x}}
    \int_{-x\log(t_j+t_f+x)}^{x\log(t_j+t_f+x)}
    \frac{\log (t_j+t_f+u)}{1+|u|}\dd u \Big).
\end{multline}
  Applying the above estimate twice we have for $z\geq (\log(t_j+t_f))^3$ that
  \begin{equation}
    |\sum_{(\log (t_f+t_j))^3<p\leq z}
    \frac{\lambda_{\sym^4f}(p)
    \lambda_{\sym^2 u_j}(p)}{p}|
    \ll 1.
  \end{equation}
  Assuming GRC, for $y\leq (\log (t_f+t_j))^3$ we have
  \begin{equation}
   |\sum_{p\leq y}
    \frac{\lambda_{\sym^4f}(p)
    \lambda_{\sym^2 u_j}(p)}{p}|
    \ll
    \log\log\log (t_f+t_j).
  \end{equation}
  This completes the proof.
\end{proof}

Using Lemma \ref{lemma:sumofcoefficients}, for $2\leq y\leq x$, $\ell_1,\ell_2,\ell_3>0$ and distinct Hecke--Maass forms $f,g$, we have that
\begin{multline}\label{eqn:sumofsumoverp}
  \sum_{y<p\leq x}\frac{(\ell_1+\ell_2\lambda_{\sym^2 f}(p)+\ell_3\lambda_{\sym^2 g}(p))^2}{p}
  =(\ell_1^2+\ell_2^2+\ell_3^2)\log\frac{\log x}{\log y}+O(\log\log\log (t_f+t_g)).
\end{multline}

\subsection{Bounds for moments of Dirichlet polynomials}

\begin{lemma}\label{lemma:sumusesummationformula}
  Let $r\in\mathbb{N}$. Then for $x\leq X^{\frac{1}{10r}}$ and real numbers $a_p\ll p^{\frac{1}{2}-\delta}$ for some $\delta >0$, we have that

  \begin{multline}
    \sum_{j} \frac{\Phi_{X,Y,M}(t_j)}{L(1,\sym^2u_j)}
    \Big(\sum_{p\leq x}
    \frac{a_p\lambda_{u_j}(p)}
    {p^{\frac{1}{2}}}\Big)^{2r}
    \ll G(X,Y,M)\frac{(2r)!}{r!2^r}
    \Big(\sum_{p\leq x}\frac{a_p^2}{p}\Big)^r
    \\
    +\frac{2^{2r}(2r)!}{r!}
    X^{1+\varepsilon}
    x^{2r(\frac{1}{2}-\delta)}.
  \end{multline}
\end{lemma}

\begin{proof}
  Using $\lambda_{u_j}(p^l)=\sum_{0\leq m\leq l}\alpha_j(p)^{m}\beta_j(p)^{l-m}$,
  we have
  \begin{equation*}
    \lambda_{u_j}(p)^k=(\alpha_j(p)+\beta_j(p))^k
    =\sum_{\substack{0\leq l\leq k\\l\equiv k\pmod{2}}}D_{k,l}\lambda_{u_j}(p^l),
  \end{equation*}
  where
  \begin{equation*}
  D_{k,l}=\frac{k!}
  {(\frac{k+l}{2})!(\frac{k-l}{2})!}-\sum_{0< m\leq \frac{k-l}{2}}D_{k,l+2m},
  \end{equation*}
  and $D_{k,k}=1$,
  so \begin{equation*}
  D_{k,l}=\frac{k!(l+1)}
  {(\frac{k+l}{2}+1)!(\frac{k-l}{2})!}.
  \end{equation*}
  Let $a_n=\prod_{p^j\|n}a_p^j$. Then we have
  \begin{multline}\label{eqn:counting}
    \sum_{j} \frac{\Phi_{X,Y,M}(t_j)}{L(1,\sym^2u_j)}
    \Big(\sum_{p\leq x}
    \frac{a_p\lambda_{u_j}(p)}
    {p^{\frac{1}{2}}}\Big)^{2r}
    =
    \sum_{\substack{n=p_1^{e_1}\dots p_q^{e_q}\\p_1,\dots,p_q\leq x
    \\e_1+\dots+e_q=2r}}
    \frac{a_n}{n^{\frac{1}{2}}}
    \sum_{\substack{0\leq l_1\leq e_1\\l_1\equiv e_1\pmod{2}}}
    \dots
    \sum_{\substack{0\leq l_q\leq e_q\\l_q\equiv e_q\pmod{2}}}
    \\
    \frac{(2r)!\prod_{i=1}^{q}(l_i+1)}
    {\prod_{i=1}^{q}\Big( (\frac{e_i+l_i}{2}+1)!
    (\frac{e_i-l_i}{2})!\Big)}
     \sum_{j} \frac{\lambda_{u_j}(p_1^{l_1}\dots p_q^{l_q})
     \Phi_{X,Y,M}(t_j)}{L(1,\sym^2u_j)}
  \end{multline}
  By Lemma \ref{lemma:3.1}, this is equal to
  \begin{multline*}
     \sum_{\substack{n=p_1^{2f_1}\dots p_q^{2f_q}\\p_1,\dots,p_q\leq x
    \\f_1+\dots+f_q=r}}    \frac{(2r)!}{\prod_{i=1}^{q}\Big( f_i!(f_i+1)!\Big)}
    \frac{a_n}{n^{\frac{1}{2}}}G(X,Y,M)
    \\
    +O\Big(\sum_{\substack{n=p_1^{e_1}\dots p_q^{e_q}\\p_1,\dots,p_q\leq x
    \\e_1+\dots+e_q=2r}}
    \sum_{\substack{0\leq l_1\leq e_1\\l_1\equiv e_1\pmod{2}}}
    \dots
    \sum_{\substack{0\leq l_q\leq e_q\\l_q\equiv e_q\pmod{2}}}
    \frac{(2r)!\prod_{i=1}^{q}(l_i+1)}{\prod_{i=1}^{q}\Big( (\frac{e_i+l_i}{2}+1)!
    (\frac{e_i-l_i}{2})!\Big)}
    \frac{|a_n|}{n}X^{1+\varepsilon} \Big).
    \end{multline*}
Since $a_{p_1^{2f_1}\dots p_q^{2f_q}}\geq 0$, by using $(n+1)!\geq 2^{n}$, we have
  \begin{equation*}
  \frac{(2r)!}{\prod_{i=1}^{q}\Big( f_i!(f_i+1)!\Big)}\leq \frac{(2r)!}{r!}\frac{r!}
{\prod_{i=1}^{q}f_i!2^{f_i}}=
\frac{(2r)!}{r!2^r}\frac{r!}
{\prod_{i=1}^{q}f_i!},
    \end{equation*}
  \begin{equation*}
  \sum_{\substack{0\leq l_i\leq e_i\\l_i\equiv e_i\pmod{2}}}\frac{e_i!(l_i+1)}
{(\frac{e_i+l_i}{2}+1)!
(\frac{e_i-l_i}{2})!}=\frac{e_i!}
{\lceil\frac{e_i}{2}\rceil !
\lfloor\frac{e_i}{2} \rfloor !},
    \end{equation*}
    and
      \begin{equation*}
  \sum_{e_1+\dots+e_q=2r}\frac{(2r)!}{
  \prod_{i=1}^{q} e_i!}
  \frac{\prod_{i=1}^{q} e_i!}
{\lceil\frac{e_i}{2}\rceil !
\lfloor\frac{e_i}{2} \rfloor !}\leq
\sum_{q\leq 2r}2^{2q}< 2^{4r+1},
    \end{equation*}
thus
\begin{equation*}
    \sum_{j} \frac{\Phi_{X,Y,M}(t_j)}{L(1,\sym^2u_j)}
    \Big(\sum_{p\leq x}
    \frac{a_p\lambda_{u_j}(p)}
    {p^{\frac{1}{2}}}\Big)^{2r}
   \ll \frac{(2r)!}{r!2^r}\Big(\sum_{p\leq x}\frac{a_p^2}{p}\Big)^r G(X,Y,M)
    +2^{4r}
    X^{1+\varepsilon}x^{2r(\frac{1}{2}-\delta)}.
\end{equation*}
This completes the proof.
\end{proof}

Before stating the next lemma, let us introduce the following notion.
For $2\leq y\leq x$, let
\begin{equation}
  \mathcal{P}(t_j;x,y)
  =\sum_{p\leq y}
  \frac{(\ell_1+\ell_2\lambda_{\sym^2 f}(p)+\ell_3\lambda_{\sym^2 g}(p))\lambda_{u_j}(p)}
  {p^{\frac{1}{2}+\frac{1}{\log x}}}
  (1-\frac{\log p}{\log x}),
\end{equation}
and $\mathcal{A}_{X,Y}(V;x)=\#\{X<t_j\leq X+Y: \mathcal{P}(t_j;x,x)>V\}$.
Also, define
\begin{equation}
\sigma(X)^2=(\ell_1^2+\ell_2^2+\ell_3^2)
\log\log (X+t_g).
\end{equation}

\begin{lemma}\label{lemma:AXY}
 Assume GRH and GRC.
  Let $C\geq 1$ be fixed and $\varepsilon >0$ be sufficiently small.
  With the above notation, we have for
  $\sqrt{\log\log X}\leq V\leq C\log (X+t_g)/\log\log (X+t_g)$ that
  \begin{equation}
    \mathcal{A}_{X,Y}(V;
    X^{\frac{1}{\varepsilon V}})
    \ll
    G(X,Y,M)\Big(e^{-\frac{(1-2\varepsilon)V^2}
    {2\sigma(X)^2}} \log\log X
    +e^{-\frac{\varepsilon}{11}V\log V}\Big) +XM\sqrt{\log X}.
  \end{equation}
\end{lemma}

\begin{proof}
  We assume throughout that $\sqrt{\log\log X}\leq V\leq C\frac{\log (X+t_g)}{\log\log (X+t_g)}$.
  Set $x=\min\{(X+t_g)^{\frac{1}{\varepsilon V}},Y^{\frac{16-0.0001}{9\varepsilon V}}\}$
  and let
  $z=x^{\frac{1}{\log\log (X+t_g)}}$.
  Write $\mathcal{P}(t_j;x,x)
  =\mathcal{P}_1(t_j)+\mathcal{P}_2(t_j)$ where $\mathcal{P}_1(t_j)=\mathcal{P}(t_j;x,z)$.
  Also, let $V_1=(1-\varepsilon)V$ and $V_2=\varepsilon V$.
  If $\mathcal{P}(t_j;x,x)>V$ then
  \begin{equation}\label{eqn:case1}
    \mathcal{P}_1(t_j)>V_1,
  \end{equation}
  or
    \begin{equation}\label{eqn:case2}
    \mathcal{P}_2(t_j)>V_2.
  \end{equation}
  Using Lemma \ref{lemma:sumusesummationformula} and \eqref{eqn:sumofsumoverp}, we have for $r\leq \frac{\varepsilon V}{10}\log\log (X+t_g)$ that the number of $X< t_j\leq X+Y$ for which \eqref{eqn:case1} holds is bounded by
  \begin{multline}
    \frac{1}{V_1^{2r}}
    \sum_{X< t_j\leq X+Y}
    \mathcal{P}_1(t_j)^{2r}
    \ll
     \frac{\log\log X}{V_1^{2r}}
    \sum_{t_j}
     \frac{\Phi_{X,Y,M}(t_j)}
     {L(1,\sym^2u_j)}
    \mathcal{P}_1(t_j)^{2r}
    +XM\sqrt{\log X}
    \\ \ll
    G(X,Y,M)\log\log X\frac{(2r)!}{V_1^{2r}r!2^r}
    (\sigma(X)(1+o(1)))^{2r}
    +XM\sqrt{\log X}
    \\ \ll G(X,Y,M)\log\log X\Big(
    \frac{2r\sigma(X)^2(1+o(1))}{V_1^2 e}\Big)^r
    +XM\sqrt{\log X},
  \end{multline}
  where in the first step we applied $L(1,\sym^2u_j)\ll (\log\log X)^3$ under GRH and GRC, and in the third step we applied Stirling's formula.
  In the range $V\leq \frac{\varepsilon }{10}\sigma(X)^2\log\log (X+t_g)$ we set $r=\lfloor \frac{V_1^2}{2\sigma(X)^2} \rfloor$ and for larger $V$ we set $r=\lfloor \frac{\varepsilon V}{10} \rfloor$.
  Hence we have
  \begin{multline*}
    \# \{ X<t_j\leq X+Y: \mathcal{P}_1(t_j)>V_1 \}
    \\ \ll
    G(X,Y,M)\Big(
    e^{-(1-2\varepsilon)\frac{V^2}
    {2\sigma(X)^2}} \log\log X
    +e^{-\frac{\varepsilon}{11} V\log V}
    \Big) +XM\sqrt{\log X}.
  \end{multline*}
It remains to bound the number of $X<t_j\leq X+Y$ for which \eqref{eqn:case2} holds.
Take $r=\lfloor \frac{\varepsilon V}{10} \rfloor$.
As before, we use Lemma \ref{lemma:sumusesummationformula} and \eqref{eqn:sumofsumoverp} to bound this quantity by
\begin{multline}
  \frac{1}{V_2^{2r}}
    \sum_{X< t_j\leq X+Y}
    \mathcal{P}_2(t_j)^{2r}
    \ll
     \frac{\log\log X}{V_2^{2r}}
    \sum_{t_j}
     \frac{\Phi_{X,Y,M}(t_j)}
     {L(1,\sym^2u_j)}
    \mathcal{P}_2(t_j)^{2r}
    +XM\sqrt{\log X}
    \\ \ll
    G(X,Y,M)\log\log X\frac{(2r)!}{r!}
    \Big(\frac{C}{V_2^2}\log\log\log (X+t_g)
    \Big)^r +XM\sqrt{\log X}
    \\ \ll
    G(X,Y,M)e^{-\frac{\varepsilon}{11}V\log V} +XM\sqrt{\log X}.
\end{multline}
Then we complete the proof.
\end{proof}

\subsection{Proof of Proposition \ref{prop:mixedmonents1}}
  Using \eqref{eqn:heckerelation} and bounding the sum over $p^n\leq x$ with $n\geq 3$ we have
  \begin{multline}\label{eqn:Lambdasumuj}
\sum_{p^n\leq x}
    \frac{\Lambda_{u_j}(p^n)}
    {np^{n(\frac{1}{2}+\frac{1}{\log x})}}
    \frac{\log \frac{x}{p^n}}
    {\log x}
    =\sum_{p\leq x}
    \frac{\lambda_{u_j}(p)}
    {p^{\frac{1}{2}+\frac{1}{\log x}}}
    \frac{\log \frac{x}{p}}{\log x}
     +\frac{1}{2} \sum_{p\leq \sqrt{x}}
    \frac{\lambda_{\sym^2 u_j}(p)-1}
    {p^{1+\frac{2}{\log x}}}
    \frac{\log \frac{x}{p^2}}{\log x}
    +O(1),
  \end{multline}
  \begin{multline}\label{eqn:Lambdasumsym2fuj}
    \sum_{p^n\leq x}
    \frac{\Lambda_{\sym^2 f \times
    u_j}(p^n)}
    {np^{n(\frac{1}{2}+\frac{1}{\log x})}}
    \frac{\log \frac{x}{p^n}}
    {\log x}
    =\sum_{p\leq x}
    \frac{\lambda_{\sym^2 f}(p)
    \lambda_{u_j}(p)}
    {p^{\frac{1}{2}+\frac{1}{\log x}}}
    \frac{\log \frac{x}{p}}{\log x}
    \\ +\frac{1}{2} \sum_{p\leq \sqrt{x}}
    \frac{(\lambda_{\sym^4 f}(p)-\lambda_{\sym^2 f}(p)+1)
    (\lambda_{\sym^2 u_j}(p)-1)}
    {p^{1+\frac{2}{\log x}}}
    \frac{\log \frac{x}{p^2}}{\log x}
    +O(1),
  \end{multline}
  and
   \begin{multline}\label{eqn:Lambdasumsym2guj}
    \sum_{p^n\leq x}
    \frac{\Lambda_{\sym^2 g \times
    u_j}(p^n)}
    {np^{n(\frac{1}{2}+\frac{1}{\log x})}}
    \frac{\log \frac{x}{p^n}}
    {\log x}
    =\sum_{p\leq x}
    \frac{\lambda_{\sym^2 g}(p)
    \lambda_{u_j}(p)}
    {p^{\frac{1}{2}+\frac{1}{\log x}}}
    \frac{\log \frac{x}{p}}{\log x}
    \\
    +\frac{1}{2} \sum_{p\leq \sqrt{x}}
    \frac{(\lambda_{\sym^4 g}(p)-\lambda_{\sym^2 g}(p)+1)
    (\lambda_{\sym^2 u_j}(p)-1)}
    {p^{1+\frac{2}{\log x}}}
    \frac{\log \frac{x}{p^2}}{\log x}
    +O(1).
  \end{multline}
By using Lemma \ref{lemma:sumofcoefficients},
the second sum of \eqref{eqn:Lambdasumuj} equals
\begin{equation}\label{eqn:secondterms1}
  -\frac{1}{2}\log\log x+O(\log \log \log t_j),
\end{equation}
the second sum of
  \eqref{eqn:Lambdasumsym2fuj} equals
  \begin{equation}\label{eqn:secondterms2}
  -\frac{1}{2}\log\log x+O(\log \log \log (t_j+t_f)),
\end{equation}
and the second sum of
  \eqref{eqn:Lambdasumsym2guj} equals
  \begin{equation}\label{eqn:secondterms3}
  -\frac{1}{2}\log\log x+O(\log \log \log (t_j+t_g)).
\end{equation}

  Let \begin{equation}
    \mu(X)=(-\frac{1}{2}+\varepsilon)
  (\ell_1+\ell_2+\ell_3)
  \log\log (X+t_g).
  \end{equation}
  Also, define
  \begin{equation}
    \mathcal{L}(t_j)=
    L(1/2, u_j)^{\ell_1}
    L(1/2, \sym^2f\times u_j)^{\ell_2}
    L(1/2, \sym^2g \times u_j)^{\ell_3},
  \end{equation}
  $D_{t}\ll t$ be the dimension for the eigenspace with spectral parameter $t_j=t$, and $\mathcal{B}_{X,Y}(V)=\#\{X<t_j\leq X+Y: \log \mathcal{L}(t_j)>V, u_j\ncong  f,g\}$.
  Clearly,
  \begin{multline}
    \sum_{X<t_j\leq X+Y}
    \mathcal{L}(t_j)
    =-\int_{\mathbb{R}}
    e^V \dd \mathcal{B}_{X,Y}(V)
    +\delta_{X<t_f\leq X+Y}D_{t_f}\mathcal{L}(t_f)
    +\delta_{X<t_g\leq X+Y}D_{t_g}\mathcal{L}(t_g)
    \\ =\int_{\mathbb{R}}
    e^V \mathcal{B}_{X,Y}(V)\dd V
    +\delta_{X<t_f\leq X+Y}D_{t_f}\mathcal{L}(t_f)
    +\delta_{X<t_g\leq X+Y}D_{t_g}\mathcal{L}(t_g)
    \\ =e^{\mu(X)}\int_{\mathbb{R}}
    e^V \mathcal{B}_{X,Y}(V+\mu(X))
    \dd V
    +\delta_{X<t_f\leq X+Y}D_{t_f}\mathcal{L}(t_f)
    +\delta_{X<t_g\leq X+Y}D_{t_g}\mathcal{L}(t_g).
  \end{multline}
  Note that $\log \mathcal{L}(t_j)
  \leq C \log(X+t_g)/\log\log(X+t_g)$ for some $C>1$, which follows from using Lemma \ref{lemma:LogLfunctions}.
  So we only need to consider $\sqrt{\log\log X}
  \leq V\leq C\frac{\log (X+t_f)}
  {\log\log(X+t_f)}$ in the integral above, and for smaller $V$ we use the trivial bound $\mathcal{B}_{X,Y}(V)\leq G(X,Y,M)+cXM\sqrt{\log X}$.
  Let $x=(X+t_g)^{\frac{1}{\varepsilon V}}$.
  For $\sqrt{\log\log X}
  \leq V\leq (\log\log (X+t_g))^4$,
  \begin{equation*}
   -\frac{1}{2}
  (\ell_1+\ell_2+\ell_3)
  \log\log x
  +O(\log\log\log (t_g+t_j))
  \leq \mu(X).
  \end{equation*}
  It follows from Lemma \ref{lemma:LogLfunctions}, \eqref{eqn:secondterms1},
  \eqref{eqn:secondterms2}, and
  \eqref{eqn:secondterms3} that
  \begin{equation*}
    \mathcal{B}_{X,Y}(V+\mu(X))
    \leq \mathcal{A}_{X,Y}
    (V(1-2\varepsilon);x)
  \end{equation*}
  provided $\sqrt{\log\log X}
  \leq V\leq (\log\log (X+t_g))^4$.
  For $V\geq (\log\log (X+t_g))^4$
  the above inequality is also true since in this range $V+\mu(X)=V(1+o(1))$.
  Hence, combining estimates and applying Lemma \ref{lemma:AXY} there exists an absolute constant $C>0$ such that
  \begin{multline}
    \sum_{X<t_j\leq X+Y}
    \mathcal{L}(t_j)
    \\ \ll
    G(X,Y,M) e^{\mu(X)}
    \int_{\sqrt{\log\log X}}
    ^{C\frac{\log (X+t_f)}
  {\log\log(X+t_f)}}
    e^V \Big(e^{-\frac{(1-\epsilon)V^2}
    {2\sigma(X)^2}} \log\log X
    +e^{-\epsilon V\log V}\Big)
    \dd V
    \\+ XM(X+t_g)^{\frac{\varepsilon^2}{100}}
    \\ \ll G(X,Y,M)
    (\log (X+t_g))^\varepsilon
    e^{\mu{X}+\frac{\sigma(X)^2}{2}}
    +XM(X+t_g)^{\frac{\varepsilon^2}{100}}
    \\ \ll
    G(X,Y,M)(\log (X+t_g))
    ^{\frac{\ell_1(\ell_1-1)}{2}
    +\frac{\ell_2(\ell_2-1)}{2}
    +\frac{\ell_3(\ell_3-1)}{2}+\varepsilon}
    +XM(X+t_g)^{\frac{\varepsilon^2}{100}},
  \end{multline}
  where in the last step we have used the identity
  \begin{equation*}
    \int_{\mathbb{R}} e^{-\frac{x^2}{2\sigma^2}+x}\dd x
    =\sqrt{2\pi} \sigma e^{\frac{\sigma^2}{2}}.
  \end{equation*}
  This completes the proof.


\section*{Acknowledgements}

The authors want  to thank  Peter Humphries,  Steve Lester, Jianya Liu, and Ze\'ev Rudnick for helpful discussions.
%
Shenghao Hua thanks CSC for its financial support, and EPFL's TAN group for its host.
This material is based upon work supported
by the Swedish Research Council under grant no. 2021-06594
while Bingrong Huang was in residence at Institut Mittag-Leffler in Djursholm, Sweden
during the spring semester of 2024.


\end{document}